\documentclass[draft, 12pt]{article}

\usepackage{amsthm}
\usepackage{amsmath}
\usepackage{amssymb}
\usepackage{amscd}                 
\usepackage[all]{xy}               
\usepackage{stmaryrd}              

\theoremstyle{plain}

\newtheorem{thm}{Theorem}[section]
\newtheorem{cor}[thm]{Corollary}
\newtheorem{lem}[thm]{Lemma}

\newtheorem{prop}[thm]{Proposition}

\theoremstyle{definition}
\newtheorem{defn}[thm]{Definition}
\newtheorem{exmp}[thm]{Example}
\newtheorem{rem}[thm]{Remark}
\newtheorem{prob}[thm]{Problem}

\newcommand\Q{{\mathbb Q}}
\newcommand\Z{{\mathbb Z}}
\newcommand{\Gal}{\mathop{\mathrm{Gal} }\nolimits}
\newcommand{\GL}{\mathop{\mathrm{GL} }\nolimits}
\newcommand{\Sp}{\mathop{\mathrm{Sp} }\nolimits}
\newcommand{\GSp}{\mathop{\mathrm{GSp} }\nolimits}
\newcommand\Oh{\mathcal{O}}
\newcommand\F{\mathbb{F}}
\newcommand\fm{\mathfrak{m}}
\newcommand\wild{\mathrm{w}}
\newcommand\tame{\mathrm{t}}
\newcommand\Id{\mathrm{Id}}
\newcommand\Frob{\mathrm{Frob}}
\newcommand{\Spec}{\mathop{\mathrm{Spec} }\nolimits}
\newcommand{\card}{\mathop{\mathrm{card} }\nolimits}
\newcommand{\Aut}{\mathop{\mathrm{Aut} }\nolimits}
\newcommand\SL{\mathrm{SL}}
\newcommand\fp{\mathfrak{p}}

\newcommand\calP{\mathcal{P}}
\newcommand\calA{\mathcal{A}}
\newcommand\calC{\mathcal{C}}
\newcommand\End{\mathrm{End}}
\newcommand\unr{\mathrm{unr}}
\newcommand\disc{\mathrm{disc}}

\title{Tame Galois realizations of $\GSp_4(\F_{\ell})$ over $\Q$}
\date{}

\author{Sara Arias-de-Reyna and  N\'uria Vila}

\begin{document}
\maketitle

\footnotetext{
Research supported by a FPU predoctoral grant AP-20040601 of the MEC
and partially supported by MEC grant MTM2006-04895.}

\begin{abstract}
In this paper we obtain realizations of the $4$-dimensional general
symplectic group over a prime field of characteristic $\ell>3$ as
the Galois group of a tamely ramified Galois extension of $\Q$. The
strategy is to consider the Galois representation $\rho_{\ell}$
attached to the Tate module at $\ell$ of a suitable abelian surface.
We need to choose the abelian varieties carefully in order to ensure
that the image of $\rho_{\ell}$ is large and simultaneously maintain
a control on the ramification of the corresponding Galois extension.
We obtain an explicit family of curves of genus $2$ such that the
Galois representation attached to the $\ell$-torsion points of their
Jacobian varieties provide tame Galois realizations of the desired
symplectic groups.
\end{abstract}

\section{Introduction}

A central problem in number theory is the study of the structure of
the absolute Galois group $G_{\Q}=\Gal(\overline{\Q}/\Q)$. The
Inverse Galois Problem over $\Q$, first considered by D. Hilbert,
can be reformulated as a question about the finite quotients of this
absolute Galois group. In spite of the many efforts made to solve
the Inverse Galois Problem, it still remains open. Of course, these
attempts have not been fruitless, and there are many finite groups
that are known to be Galois groups over $\Q$. For an idea of the
progress made, the reader can leaf through \cite{Topics in Galois
Theory} or \cite{Malle-Matzat}.

Assume that a finite group $G$ can be realized as a Galois group
over $\Q$, say $G\simeq \Gal(K_1/\Q)$, where $K_1/\Q$ is a finite
Galois extension. But perhaps we are interested in field extensions
with some ramification property, and $K_1/\Q$ does not satisfy it.
We can ask whether there exists some other finite Galois extension,
$K_2/\Q$, with Galois group $G$ and enjoying this additional
property. In this connection, several variants of the Inverse Galois
Problem have been studied.

This paper is concerned with the following problem, posed by Brian
Birch (cf. \cite{Birch}, Section 2).

\begin{prob}[Tame Inverse Galois Problem]\label{ProblemaBirch} Given a
finite group G, is there a tamely ramified Galois extension $K/\Q$
with $\Gal(K/\Q)\simeq G$?
\end{prob}

This problem has already been dealt with for certain families of
groups. It is  known that solvable groups, all symmetric groups
$S_n$, all alternating groups $A_n$, the Mathieu groups $M_{11}$ and
$M_{12}$, and their finite central extensions are realizable as the
Galois group of a tamely ramified extension (cf.
\cite{Kluners-Malle}, \cite{Plans-Vila2003}, \cite{Plans2003}).

One way to deal with the Inverse Galois Problem, and eventually with
Problem \ref{ProblemaBirch}, is to consider continuous Galois
representations of the absolute Galois group of the rational field.
Let $V$ be a finite dimensional vector space over a finite field
$\F$, and consider a continuous representation of $G_{\Q}$ (endowed
with the Krull topology) in the group of automorphisms of $V$ (with
the discrete topology), say
\begin{equation*}\rho: G_{\Q}\rightarrow \GL(V).\end{equation*}
Then
\begin{equation*}\mathrm{Im}\rho\simeq G_{\Q}/\ker \rho\simeq
\Gal(K/\Q),\end{equation*} where $K/\Q$ is a finite Galois
extension.

In other words, a continuous Galois representation of $G_{\Q}$
provides a realization of $\mathrm{Im}\rho$ as a Galois group over
$\Q$. This strategy has already been used to address the Inverse
Galois Problem for families of linear groups (cf. \cite{Vila2004},
\cite{Dieulefait-Vila2000}, \cite{Dieulefait-Explicit
determination}, \cite{Wiese}, \cite{Dieulefait-Wiese},
\cite{KhareLarsenSavin}).

In this paper we will study the Galois representations that arise
through the action of the absolute Galois group $G_{\Q}$ on the
$\ell$-torsion points of abelian surfaces.

For each prime number $p$, let us fix an immersion
$\overline{\Q}\hookrightarrow \overline{\Q}_{p}$. This induces an
inclusion of Galois groups $\Gal(\overline{\Q}_p/ \Q_p)\subset
G_{\Q}$. Inside $\Gal(\overline{\Q}_{p}/ \Q_p)$ we can consider the
inertia subgroup $I_p=\Gal(\overline{\Q}_p/\Q_{p, \mathrm{unr}})$
and the wild inertia subgroup $I_{p,
\wild}=\Gal(\overline{\Q}_p/\Q_{p, \tame})$, where $\Q_{p,
\mathrm{unr}}$ and $\Q_{p, \tame}$ denote the maximal unramified
extension and the maximal tamely ramified extension of $\Q_p$,
respectively. A prime $p$ is unramified (respectively tamely
ramified) in the Galois extension $K/\Q$ if and only if $\rho(I_p)=
\{\Id\}$ (respectively $\rho(I_{p,\wild}) =\{\Id\}$). Therefore, if
we want a Galois representation $\rho$ to yield a tamely ramified
Galois extension $K/\Q$, we will have to ensure that $\rho(I_{p,
\wild})=\{\Id\}$ for all prime numbers $p$.

In \cite{Ariasdereyna-Vila2009}, the authors apply this strategy to
solve Problem \ref{ProblemaBirch} for the family of $2$-dimensional
linear groups $\GL_2(\F_{\ell})$. More precisely, they consider the
Galois representations attached to the $\ell$-torsion points of
elliptic curves, and construct explicitly elliptic curves such that
this representation is surjective and moreover the image of the wild
inertia group $I_{p, \wild}$ is trivial for all primes $p$. In this
paper we will extend these results to the case of curves of genus
$2$. Namely, we will prove that, for all prime number $\ell\geq 5$
there exist (infinitely many) genus $2$ curves $C$ such that the
Galois representation attached to the $\ell$-torsion points of the
Jacobian of $C$ provides a tamely ramified Galois realization of
$\GSp_4(\F_{\ell})$ (cf. Theorem \ref{TeoremaPrincipal}).

We want to thank Luis Dieulefait for many enlightening
conversations.

\section[Semistable Reduction]{The Action of the Inertia
Group}\label{p distinto de l}

Let $A/\Q$ be an abelian variety of dimension $n$, and let us fix a
prime $\ell$. Consider the Galois representation
\begin{equation*}\rho_{\ell}:G_{\Q}\rightarrow
\Aut(A[\ell])\simeq \GL_{2n}(\F_{\ell}).\end{equation*} attached to
the $\ell$-torsion points of $A$. This representation gives rise to
a realization of $\mathrm{Im}\rho_{\ell}$ as Galois group over $\Q$,
say $\mathrm{Im}\rho_{\ell}\simeq \Gal(K/ \Q)$, where $K$ is the
number field fixed by $\ker\rho_{\ell}$. In this section we are
interested in finding conditions over the variety $A$ that ensure
that $K/ \Q$ is tamely ramified.

Note that it suffices to control the image by $\rho_{\ell}$ of
$I_{p, \wild}$ just for a finite quantity of primes $p$. Indeed,
$I_{p, \wild}$  is a pro-$p$-group, so that if $p$ does not divide
the cardinal of $\GL_{2n}(\F_{\ell})$, the cardinal of
$\mathrm{Im}(I_{p, \wild})$ must be $p^0=1$.

We will first address the problem of obtaining tame ramification at
a prime $p\not=\ell$. Let us place ourselves in a more general
setting. Let $F$ be a number field, and let us consider an abelian
variety $A/F$. Let  $\fp$ be a prime ideal of the ring of integers
$\Oh_F$ of $F$. We begin by recalling what it means for $A$ to have
semistable reduction at $\fp$ (see \cite{Diophantine Geometry}, $\S$
A.9.4.).

\begin{defn} We will
say that $A/F$ has \emph{semistable reduction} at a prime $\fp\in
\Spec\Oh_F$ if the connected component of the special fibre of the
N\'eron model of $A$ at $\fp$, $\calA_{\fp}^0$, is the extension of
an abelian variety by a torus.

\end{defn}

The kind of reduction of an abelian variety at a prime $\fp$ is
reflected in the action of the inertia group at $\fp$ on the Tate
module of the variety. For instance, if the variety has good
reduction, the N\'{e}ron-Ogg-Shafarevich criterion ensures that the
inertia group acts trivially on the Tate module. Moreover, a result
of Grothendieck characterizes the case when the reduction is
semistable in terms of the action of the inertia group. As a
consequence, we have the following result.

\begin{thm}\label{tame at p} Let $A/\Q$ be an abelian variety, and
let $\ell, p$ be two different prime numbers. Assume that $A$ has
semistable reduction at $p$. Then the image of the wild inertia
group $I_{p, \wild}$ by the Galois representation $\rho_{\ell}$ is
trivial.
 \end{thm}

\begin{proof} Let us see that $I_{p,
\wild}$ acts trivially on $T_{\ell}(A)$. By Proposition 3.5 of
\cite{SGA7}, we know that there exists a submodule $T'\subset
T_{\ell}(A)$, fixed by the action of $I_p$, and such that $I_p$ acts
trivially on both $T'$ and $T_{\ell}(A)/T'$. Taking a suitable basis
of $T_{\ell}(A)$,  $I_p$ acts through a matrix of the shape
$\left(\begin{array}{cc} \mathrm{Id}_r& *\\
                  0 & \mathrm{Id}_s
\end{array}\right)$.

But the order of such a matrix is a divisor of $\ell$. Since $I_{p,
\wild}$ is a pro-$p$-group, its elements  all act as the identity.
\end{proof}

To control the image of the wild inertia group at the prime $\ell$
is a much more subtle matter. The first author addresses this
problem in \cite{Paper1}, and obtains a condition that guarantees
that the action is trivial (Theorem 3.3). Let us recall the
statement here.

\begin{thm}\label{InerciaEnl} Let $\ell$ be a prime number, and let $A/\Q$ be an abelian variety of dimension $n$ with good supersingular reduction at $\ell$. Call $\mathbf{F}$ the formal group law attached to $A$ at $\ell$, $v_{\ell}$ the $\ell$-adic valuation on $\overline{\Q}_{\ell}$ and $V$ the group of $\ell$-torsion points attached to $\mathbf{F}$.
Assume that there exists a positive $\alpha\in \Q$ such that, for
all non-zero $(x_1, \dots, x_n)\in V$, it holds that
$\min\{v_{\ell}(x_i): 1\leq i\leq n\}=\alpha$.

Then the image of the wild inertia group by the Galois
representation attached to the $\ell$-torsion points of $A$ is
trivial.
\end{thm}

This result will allow us to hold in check the action of the wild
inertia group at $\ell$.

\section{Image of the Representation}\label{Image of the
representation}

In the previous section we considered the tameness condition. Our
aim in this section is to obtain some control on the image of the
representation.

Let us fix a prime $\ell$, let $A/\Q$ be an abelian variety of
dimension $n$, and let us denote by $\rho_{\ell}:G_{\Q}\rightarrow
\GL_{2n}(\F_{\ell})$ the Galois representation attached to the
$\ell$-torsion points of $A$. Assume $A$ is principally polarized.
Then the Weil pairing gives rise to a non-degenerated symplectic
form on the group of $\ell$-torsion points of $A$, $\langle \cdot ,
\cdot \rangle: A[\ell]\times A[\ell]\rightarrow \F_{\ell}^*$.
Furthermore, the elements of the Galois group $G_{\Q}$ behave well
with respect to this pairing. This compels the image of the
representation to be contained in the general symplectic group
$\GSp_{2n}(\F_{\ell})$.

Now a well-known result of Serre states that, if the principally
polarized abelian variety has the endomorphism ring equal to $\Z$,
and furthermore its dimension is either odd or equal to $2$ or $6$,
then the image of the representation $\rho_{\ell}$ is the whole
general symplectic group for all but finitely many primes $\ell$
(see Theorem 3 of 137 in \cite{Ouvres}).

We are particularly interested in the case of abelian surfaces over
$\Q$. In this case, the result of Serre boils down to:

\begin{thm} Let $A/\Q$ be an abelian surface, principally polarized,
such that $\End_{\overline{\Q}}(A)=\Z$. Then, for all but finitely
many primes $\ell$, it holds that
\begin{equation*}\mathrm{Im}\rho_{\ell}=\GSp_4(\F_{\ell}).\end{equation*}

\end{thm}

We will take  the explicit results of P. Le Duff \cite{LeDuff} as
our starting point. We will sketch some of his reasoning in order to
make use of it later on, and then we will dwell upon some specific
points where we need to modify it.

The method of Le Duff rests upon finding a certain set of elements
in the image of the representation which generate the whole
$\GSp_4(\F_{\ell})$. Therefore, at the core of the method lies a
theorem about generators of this group. More specifically, the main
result upon which Le Duff builds his method is the following (see
Theorem 2.7 of \cite{LeDuff}):

\begin{prop}\label{Prop LeDuff} The symplectic group $\Sp_4(\F_{\ell})$ is generated by
a transvection and an element whose characteristic polynomial is
irreducible.
\end{prop}

It is easy to see that, if the image of the Galois representation
contains the subgroup $\Sp_4(\F_{\ell})$, then it necessarily
contains the whole group $\GSp_4(\F_{\ell})$, since the composition
of the representation with the projection
$\GSp_4(\F_{\ell})\rightarrow \GSp_4(\F_{\ell})/\Sp_4(\F_{\ell})$
equals the cyclotomic character, which is surjective. Therefore, the
problem boils down to finding two elements in
$\mathrm{Im}\rho_{\ell}$, satisfying that one is a transvection and
that the other has an irreducible characteristic polynomial.

This is the way we approached the matter at first, but in order to
produce an element with irreducible characteristic polynomial in
$\mathrm{Im}\rho_{\ell}$ we needed to make use of a conjecture of
Hardy and Littlewood (namely Conjecture (F) of
\cite{Hardy-Littlewood}). Following a suggestion by L. Dieulefait we
have managed to devise a way to ensure a large image without
resorting to this conjecture.

Let us  consider the following result (Theorem 2.2 of
\cite{LeDuff}).

\begin{thm}\label{TresCasos} Let $G$ be a proper subgroup of $\Sp_4(\F_{\ell})$, and
assume that $G$ contains a transvection. Then one of the following
three assertions holds.
\begin{itemize}
\item[(1)] $G$ stabilizes a hyperplane and a line belonging to it.
\item[(2)] $G$ stabilizes a totally isotropic plane. \item[(3)] The elements
of $G$ stabilize or exchange two orthogonal supplementary
non-totally isotropic planes.
\end{itemize}
\end{thm}

\begin{rem}If $G$ is a subgroup of $\Sp_4(\F_{\ell})$ which contains an element with
irreducible characteristic polynomial, it cannot satisfy any of the
three assertions of the theorem (see Theorem 2.7 of \cite{LeDuff}).
Therefore Proposition \ref{Prop LeDuff} is an easy consequence of
this result.
\end{rem}

Our strategy will be the following: for each of the three
assertions, we shall ensure the existence of an element of
$G=\mathrm{Im}\rho_{\ell}$, contained in $\Sp_4(\F_{\ell})$, which
does not satisfy it. In this way, we will prove that $G$ cannot be a
proper subgroup of $\Sp_4(\F_{\ell})$. Instead of asking directly
that there exists an element with irreducible characteristic
polynomial, which rules out the three possibilities at once, we will
require that there are elements such that the corresponding
characteristic polynomial decomposes in different ways.

\begin{rem}\label{SubgruposMaximales} The second assertion in the
above theorem occurs when $G$ is contained in a maximal parabolic
subgroup. In this case, it is easy to check that, choosing a
suitable symplectic basis, this maximal subgroup consists of
matrices of the form $\begin{pmatrix}A & * \\ 0 &
(A^{-1})^t\end{pmatrix}$, where $(A^{-1})^t$ denotes the transpose
of the inverse matrix of $A$ (cf. the remark following the proof of
Theorem 2.2 in \cite{LeDuff}). On the other hand, if the third
assertion holds, then the elements of $G$ leave two supplementary
orthogonal non-totally isotropic planes stable, or else interchange
them (this is case (3) of Proposition 2 of \cite{Kawamura}). If an
element of $\Sp_4(\F_{\ell})$ interchanges two such planes, then it
can be seen that its trace is zero. Therefore an element which
belongs to this kind of maximal subgroup either has trace $0$ or
stabilizes two planes. Moreover, if it stabilizes two planes, it can
be expressed as $\begin{pmatrix}A & *\\ 0 & B\end{pmatrix}$ with
respect to some suitable basis, where $A$ and $B$ belong to
$\SL_2(\F_{\ell})$.
\end{rem}

Let us consider an element in $\Sp_4(\F_{\ell})$, and call its
characteristic polynomial $P(X)$. It is easy to see that this
polynomial can be written as $P(X)=X^4 + aX^3 + bX^2 + aX + 1$ for
some $a, b\in \F_{\ell}$. In turn, this implies that there exist
$\alpha, \beta\in \overline{\F}_{\ell}$ such that $P(X)$ decomposes
as $(X-\alpha)(X-\beta)(X-1/\alpha)(X-1/\beta)$ over
$\overline{\F}_{\ell}$ (for $P(\alpha)=0$ implies that
$P(1/\alpha)=0$).

Now there are essentially two ways in which such a polynomial can
break up in quadratic factors, namely
\begin{equation*}P(X)=\begin{cases}\Bigl((X-\alpha)(X-1/\alpha)\Bigr)\cdot
\Bigl((X-\beta)(X-1/\beta)\Bigr)\\
\Bigl((X-\alpha)(X-\beta)\Bigr)\cdot
\Bigl((X-1/\alpha)(X-1/\beta)\Bigr).\end{cases}\end{equation*}

The first case is labelled \emph{``unrelated'' $2$-dimensional
constituents} and the second one \emph{``related'' $2$-dimensional
constituents} in \cite{Dieulefait-Explicit determination}.

There is a nice way to discern whether the first decomposition takes
place. Namely, consider the polynomial $P_0(X)=X^2 + aX + (b-2)$.
The roots of this polynomial are precisely $\alpha + 1/\alpha, \beta
+ 1/\beta$. Therefore, if the first factorization occurs, this
polynomial is reducible and its discriminant $\Delta_0=a^2- 4b+8$ is
a square in $\F_{\ell}$. To determine whether the other
factorization takes place is more difficult. Given an element of
$\Sp_4(\F_{\ell})$ with characteristic polynomial $P(X)=X^4 + aX^3 +
bX^2 + aX + 1$, we will denote $\Delta_0(P)=a^2-4b+8$.

\begin{thm}\label{TeoremaModificado} Let $G$ be a  subgroup of $\Sp_4(\F_{\ell})$, and
assume that $G$ contains a transvection. Furthermore, assume that it
contains two elements whose characteristic polynomials, $P_1(X)$ and
$P_2(X)$, satisfy the following: denoting by $\alpha_i, 1/\alpha_i,
\beta_i, 1/\beta_i$ the four roots of $P_i(X)$, $i=1, 2$,
\begin{itemize}
\item $\alpha_1 + 1/\alpha_1, \beta_1 + 1/\beta_1\not\in
\F_{\ell}$ and $\alpha_1 + 1/\alpha_1 + \beta_1 + 1/\beta_1\not=0$.
\item $\alpha_2 + 1/\alpha_2, \beta_2 + 1/\beta_2\in \F_{\ell}$,
$\Delta_0(P_2)\not=0$ and $\alpha_2\not\in \F_{\ell}$.

\end{itemize}
Then $G$ equals $\Sp_4(\F_{\ell})$.
\end{thm}

\begin{proof}
Since $G$ contains a transvection, Theorem \ref{TresCasos} implies
that either $G$ is the whole symplectic group or else one of the
assertions (1), (2) or (3) of the theorem holds. We will see that,
in fact, none of them is satisfied.

If assertion (1) holds, then all the elements of $G$ must leave a
line invariant. But this implies that each element has one
eigenvalue which belongs to  $\F_{\ell}$. But $\alpha_1 +
1/\alpha_1$ and $\beta_1 + 1/\beta_1$ do not belong to $\F_{\ell}$.
Therefore assertion (1) does not hold.

Assume now that assertion (2) holds. Then $G$ is contained in a
group which stabilizes a totally isotropic plane. Therefore, with
respect to a suitable symplectic basis, it is contained in a
subgroup of the shape $\begin{pmatrix}A & *\\
0 &(A^{-1})^t\end{pmatrix}$ (see  Remark \ref{SubgruposMaximales}).
In particular, this implies that if $P(X)$ is the characteristic
polynomial of an element of $G$, it must factor over $\F_{\ell}$
into two polynomials of degree two. Call  the roots of one of the
factors $\alpha$ and $\beta$. Then the roots of the other factor are
$1/\alpha$ and $1/\beta$.

Let us consider the polynomial
$P_2(X)=(x-\alpha_2)(x-\beta_2)(x-1/\alpha_2)(x-1/\beta_2)$.
Labelling the roots anew if necessary, we can assume that $\alpha_2$
and $\beta_2$ are the roots of one of the quadratic factors, as
above. We can consider two cases:
\begin{itemize}
\item $\beta_2=1/\alpha_2$ or $\beta_2=\alpha_2$. In this case
$P(X)$ can be factored as $P(X)=(X^2 - AX + 1)^2$ for a certain
$A\in \F_{\ell}$. If we work out this expression, we obtain
$P(X)=X^4 + 2AX^3 + (A^2 + 2)X^2 + 2AX + 1$. Writing out $P(X)=X^4 +
a_2X^3 + b_2X^2 + a_2X + 1$ and comparing these two expressions we
obtain that $\Delta_0(P_2)=0$, which contradicts our hypotheses on
$P_2(X)$.

\item $\beta_2\not=\alpha_2, 1/\alpha_2$. In this case, the
polynomials $(X-\alpha_2)(X-1/\alpha_2)$ and
$(X-\alpha_2)(X-\beta_2)$ are different, and both are defined over
$\F_{\ell}$. Therefore, one can use the Euclid algorithm to compute
their greatest common divisor $(X-\alpha_2)$. This implies that
$\alpha_2\in \F_{\ell}$, which contradicts our hypotheses on
$P_2(X)$.
\end{itemize}

Finally, assume that Assertion (3) holds. Then any element in $G$
satisfies that either its trace is zero or it stabilizes two planes,
which are supplementary, orthogonal and are not totally isotropic
(see Remark \ref{SubgruposMaximales}). Consider again the element
with characteristic polynomial $P_1(X)$. Since it has non-zero
trace, it must stabilize two such planes. But then $P_1(X)$ should
break into two quadratic factors defined over $\F_{\ell}$. Moreover,
since the determinant of the corresponding matrix is $1$ for each of
the factors (cf. Remark \ref{SubgruposMaximales}), their independent
terms must be $1$. But this means that the factors have to be
$(X-\alpha_1)(X-1/\alpha_1)$ and $(X-\beta_1)(X-1/\beta_1)$, and we
know these polynomials are not defined over $\F_{\ell}$. Therefore
Assertion (3) cannot hold.
\end{proof}

\section{Explicit Construction}\label{Explicit}

In this section we will face the problem of constructing explicitly,
for a given prime number $\ell$, a genus $2$ curve such that the
Jacobian variety attached to it gives rise to a Galois
representation yielding a finite Galois extension $K/\Q$, tamely
ramified, with Galois group $\GSp_4(\F_{\ell})$, thus providing an
affirmative answer to the Tame Inverse Galois Problem for this
group. In the preceding sections we have worked out some statements
that give very accurate and explicit conditions for the Galois
representation attached to the $\ell$-torsion points of an abelian
surface to satisfy the desired properties. Our aim in this section
is to replace all these conditions by others, which are more
restrictive, but are simply congruences modulo powers of different
primes. We shall tackle each of the conditions separately; thus the
section will be split in several different subsections. In the last
section we shall state a theorem possessing a very explicit flavour.

\subsection{Explicit control of the ramification}\label{Section good
reduction}

Assume we have a hyperelliptic curve $C$ of genus $g$ defined over a
certain field $k$ by an equation of the shape $y^2 + Q(x)\cdot y=
P(x)$, where $P(X)$ has degree $2g+2$. The  \emph{discriminant} of
this equation is defined  as
\begin{equation*}\Delta=2^{-4(g+1)}\cdot \disc(4P(x)+Q(x)^2).\end{equation*}
It holds that if $\Delta\not=0$, the curve $C$ is smooth (see
\cite{Liu96}, $\S$ 2).

In fact, if we fix a prime $p$, the $p$-adic valuation of the
discriminant is a bound of the conductor exponent of $C$, which in
turn coincides with the conductor exponent of the Jacobian variety
attached to $C$.

Consider the genus $2$ curve defined over $\Q$ by the hyperelliptic
equation
\begin{equation*}y^2=x^6 + 1.\end{equation*}
The discriminant of this equation is $\Delta=-2^6\cdot 3^6$.
Therefore, $\Delta\not\equiv 0 \pmod p$, for $p\not=2, 3$. Now, if
we have a genus $2$ curve $C$ defined by a hyperelliptic equation
\begin{equation}\label{eq curva}y^2=f(x),\end{equation} where $f(x)\in \Z[x]$ is a
polynomial of degree $6$ such that $f(x)\equiv x^6 + 1$ modulo a
prime $p>3$, then the prime $p$ cannot divide the discriminant of
Equation \eqref{eq curva}, thus $C$ has good reduction at $p$. In
this way, we obtain a condition that we can ask a curve to satisfy
if we want it to have good reduction at a given prime $p\not=2, 3$.
For the primes $2$ and $3$ one has to require other conditions. The
following propositions provide these conditions.

\begin{prop}\label{ReductionAt3} Let $C$ be a genus $2$ curve given by the hyperelliptic equation
\begin{equation*} y^2=f(x),\end{equation*} where $f(x)=f_6 x^6 + f_5 x^5 + f_4 x^4 + f_3 x^3
+ f_2 x^2 + f_1 x + f_0\in \Z[x]$ satisfies:
\begin{equation*}\begin{cases}f_0\equiv f_6\equiv 1\pmod
3\\
         f_1\equiv f_5\equiv 0\pmod 3\end{cases}
         \begin{cases}f_2\equiv f_4\equiv 1\pmod 3\\
         f_3\equiv 0 \pmod 3.\end{cases}\end{equation*}

Then $C$ has good reduction at $p=3$.
\end{prop}

\begin{proof} The  hyperelliptic
equation $y^2=x^6 + x^4 + x^2 + 1$ has discriminant $\Delta=
-4194304 \equiv 2\quad \pmod 3$. Because of the congruence
conditions on the coefficients $f_0, \dots, f_6$, it is clear that
the discriminant of the hyperelliptic equation defining $C$ is
congruent with $\Delta$ modulo 3, thus it is not divisible by $3$.
\end{proof}

\begin{prop}\label{ReductionAt2} Let $C$ be a genus
$2$ curve given by the hyperelliptic equation
\begin{equation}\label{eq prop} y^2=f(x),\end{equation} where
$f(x)=f_6 x^6 + f_5 x^5 + f_4 x^4 + f_3 x^3 + f_2 x^2 + f_1 x +
f_0\in \Z[x]$ satisfies that:
\begin{equation*}\begin{cases}f_0\equiv f_6\equiv 1\pmod{16}\\
         f_1\equiv f_5\equiv 0 \pmod{16}\end{cases}
         \begin{cases}f_2\equiv f_4\equiv 4 \pmod{16}\\
         f_3\equiv 2\pmod{16}.\end{cases}\end{equation*}

Then the Jacobian surface attached to $C$ has either good reduction
or semistable reduction at $p=2$.
\end{prop}

\begin{proof} Let us consider the following change of variables
\begin{equation*}\begin{cases}x:= x\\
         y:=x^3 + 2y + 1.\end{cases}\end{equation*}

Applying it to Equation \eqref{eq prop}, we obtain an equation with
integer coefficients, which are congruent to those of $ y + x^3y +
y^2=x^4 + x^2 $ modulo $4$. But the discriminant of this equation
(modulo 4) is $\Delta=2$, therefore it is divisible by $2$, and once
only. Since the $2$-adic valuation of the discriminant bounds the
conductor exponent at $2$, this ensures that the Jacobian surface
attached to $C$ has either good reduction or bad semistable
reduction at $2$.
\end{proof}

In this way we can construct an abelian surface $J(C)$ with
semistable reduction at any finite set of primes $p\not=\ell$.
Concerning the prime $\ell$, in \cite{Paper1} the author gives a
very explicit condition to ensure that the wild inertia group at
$\ell$ acts trivially on the $\ell$-torsion points of the Jacobian
of a genus $2$ curve. We recall this result here.

\begin{thm}\label{Reducible} Let $\ell$ be a prime number, let $\overline{a}\in \F_{\ell}$ such that $x^2-x+\overline{a}$ divides the Deuring polynomial $H_{\ell}(x)$, and lift it to $a\in \Z$. The equation
\begin{equation*}y^2=x^6 + \frac{1-a}{a}x^4+ \frac{1-a}{a}x^2 + 1\end{equation*}
defines a genus $2$ curve such that its Jacobian variety satisfies
the hypothesis of Theorem \ref{InerciaEnl}.
\end{thm}

Moreover, in Theorem 6.4 of \cite{Paper1} this result is expanded to
cover a larger family of curves.

\begin{thm}\label{Enlargement} Let $\ell$ be a prime number, let $\overline{a}\in \F_{\ell}$ such that $x^2-x+\overline{a}$ divides the Deuring polynomial $H_{\ell}(x)$, and lift it to $a\in \Z$. Let $f_0, f_1, \dots, f_6\in \Z$ satisfy
\begin{equation*}\begin{cases}f_6\equiv f_0 \pmod{\ell^4}\\
                              f_5\equiv f_1 \pmod{\ell^4}\\
                              f_4\equiv f_2 \pmod{\ell^4}.\end{cases}\end{equation*}
Furthermore, assume that \begin{equation*}\begin{cases}f_6\equiv 1 \pmod{\ell}\\
                              f_5\equiv 0 \pmod{\ell}\end{cases}
                              \begin{cases}f_4\equiv \frac{1-a}{a} \pmod{\ell}\\
                              f_3\equiv 0 \pmod{\ell}.\end{cases}\end{equation*}

The equation $y^2=f_6x^6 +f_5x^5 +  f_4x^4 + f_3x^3 + f_2x^2 + f_1 x
+ 1$ defines a genus $2$ curve such that its Jacobian variety
satisfies the hypothesis of Theorem \ref{InerciaEnl}.
\end{thm}

\begin{rem} Note that the genus $2$ curves from theorem \ref{Reducible} have a non-hyperelliptic involution. In fact, their Jacobians are reducible. Therefore, the image of the representations attached to the $\ell$-torsion points of these Jacobians cannot be $\GSp_4(\F_{\ell})$. The enlargement of the family of curves carried out in Theorem \ref{Enlargement} is thus essential to our construction.
\end{rem}

\subsection{Control of the image: the transvection}\label{transvection}

Our aim in this section is to ensure the existence of a transvection
in the image of $G_{\Q}$ by the Galois representation $\rho_{\ell}$
attached to the Jacobian of a genus $2$ curve. Combining Proposition
1.3 and Lemma 4.1 of \cite{LeDuff} we can state the following
result:

\begin{prop}\label{prop toro}
Let $p$ be a prime number, let $C/\Q$ be a genus $2$ curve and $J$
its Jacobian. Assume that $C$ has stable reduction of type $(II)$ or
$(VI)$ at $p$.

If $\ell\not=p$ is a prime number such that it does not divide
$(\widetilde{J}_v:\widetilde{J}^0_v)$ (which denotes the order of
the group of connected components of the special fibre of the
N\'{e}ron model of $J$ at $p$), then there exists an element in the
inertia group $I_p$ such that its image by $\rho_{\ell}$ is a
transvection.
\end{prop}

Recall that if $C$ is a smooth, projective, geometrically connected
curve of genus $g\geq 2$ over a number field $K$, then $C$ has
stable reduction at a prime $\fp$ if and only if the Jacobian
variety attached to $C$ has semistable reduction at $\fp$ (see
\cite{Liu}, Remark 4.26 of chapter 10).

Proposition \ref{prop toro} leaves us a great amount of freedom to
choose $p$, and as a matter of fact we will always choose $p=5$. Of
course, this construction will not work for $\ell=5$, but this is a
minor hindrance. We just need a special construction for $\ell=5$.

We want to build a genus $2$ curve, defined over $\Q$, with stable
reduction of type (II) at $5$. Moreover, we will require that the
order of the group of connected components of the special fibre of
the N\'{e}ron model at $5$ is $1$ (thus ensuring that $\ell$ does
not divide it).

A theorem of Deligne and Mumford tells us that every smooth,
geometrically connected, projective curve defined over a local
field, say $K$, acquires stable reduction over a finite extension of
$K$ (see \cite{Liu}, Theorem 4.3 of Chapter 10). This stable
reduction can belong to one of the following types (I), (II), (III),
(IV), (V), (VI), (VII). When the curve has genus $2$, Q. Liu has
worked out a characterization of the type of (potential) stable
reduction in terms of the Igusa invariants (see \cite{Liu93}). We
will denote them by $J_2, J_4, J_6, J_8, J_{10}$, and also
$I_4:=J_2^2-2^3 3J_4$, $I_{12}:=-2^3J_4^3 + 3^2 J_2J_4J_6 - 3^3
J_6^2 - J_2^2 J_8$.

A technical remark should be made at this point. The results of Liu
are stated over a local field $K$ with separably closed residual
field $k$. We shall assume, from now on, that our curve $C$ is
defined over $\Q_{p, \unr}$ (the maximal unramified extension of
$\Q_p$, which satisfies this condition). In this way, what we shall
obtain after some reasoning is the existence of a transvection
inside the Galois group of the extension $\overline{\Q}_p/\Q_{p,
\unr}$, which in fact is the inertia group at $p$. In the
computation of the order of the group of connected components of the
special fibre of the N\'{e}ron model at $p$, the $p$-adic valuation
in $\Q_{p, \unr}$ shall come into play; but since $\Q_{p, \unr}$ is
an unramified extension of $\Q_p$, this valuation shall coincide
with the usual valuation in $\Q_p$, without any need to normalize.
Therefore, considering $\Q_{p, \unr}$ as the base field will not
give rise to any significant modification, and we shall be able to
apply the results of Liu without paying any further attention to
this point. We will just state the results concerning potential good
reduction (that is to say, type (I)) and type (II).

\begin{thm}\label{StableReductionII} Let $R$ be a discrete
valuation ring with maximal ideal $\fm$ and quotient field $K$. Let
$C/K$ be a smooth geometrically connected projective curve of genus
$2$, defined by the equation $y^2=f(x)$, where $f(x)$ is a degree
$6$ polynomial. Denote by $J_2, \dots, J_{10}$ the Igusa invariants
of $f(x)$, and denote by $\calC_{\overline s}$ the geometric special
fibre of a stable model of $C$ over some finite extension of $K$.
Then it holds:

\begin{itemize}

\item $\calC_{\overline{s}}$ is smooth if and only if
$J^5_{2i}J_{10}^{-i}\in R$ for all $i\leq 5$.

\item $\calC_{\overline{s}}$ is irreducible with a unique double
point if and only if $J_{2i}^6I_{12}^{-i}\in R$ for all $i\leq 5$
and $J_{10}^6 I_{12}^{-5}\in \fm$. If this is the case, the group
$\Phi$ of connected components of the N\'eron model  at $v$ is
isomorphic to $\Z/e\Z$, where $e=\frac{1}{6}v(J^6_{10}J^{-5}_{12})$.

\end{itemize}
\end{thm}

\begin{rem} In the first case in the theorem above, the curve $C$ is said
to have \emph{potential good reduction}, and in the second case
\emph{potential stable good reduction of type (II)}.
\end{rem}

Let us now turn our attention to a simple example:

\begin{exmp}\label{Ejemplo5}

Let us consider the curve $C$ defined by the following equation:
\begin{equation*}y^2=x^6 + x^5 + x^3 + x + 1.\end{equation*}

By using the Magma Computational Algebra System, we can compute the
Igusa invariants of $C$. We obtain the following results:

\begin{multline*}J_2=-97/4, J_4=1323/128, J_6=-14515/1024,\\
J_8=3881491/65536, J_{10}=6845/256.\end{multline*}

Recall that the last of the invariants was the discriminant of the
equation. Since $J_{10}=6845/256=2^{-8}\cdot 5\cdot 37^2$,  the only
two odd primes of bad reduction are $5$ and $37$. Thus we know that,
outside these two primes and possibly 2, the curve has good
reduction.

Let us study the type of reduction at $5$. Computing $J_{2i}^5
J_{10}^{-i}$, for  $i=1, 2, 3, 4, 5$, we see that, if $p=5$ (and
also if $p=37$), these numbers do not all belong to $\Z_p$.
Therefore, for $p=5$ (and $p=37$), the reduction of $C$ at $p$ is
not good.

Now we wish to determine if the reduction is of type (II). We have
to compute $J_{2i}^6I_{12}^{-i}$ for $i=1, 2, 3, 4, 5$. We begin
with $I_{12}$;

\begin{equation*}I_{12}=-\frac{1095163}{64}=-2^{-6}\cdot 37 \cdot
29599.\end{equation*}

Note that $p=5$ divides the discriminant of the equation, but it
does not divide $I_{12}$. And this is enough to ensure that the
reduction at $p=5$ is (potentially) stable of type (II). Moreover,
since $p=5$ divides the discriminant of the equation just once, the
reduction is indeed stable. And it turns out that
$v_5(J_{10}^6I_{12}^{-5})=6$, which means that the order of the
group of connected components of the special fibre of the N\'{e}ron
model at $p=5$ is $1$.

\end{exmp}

\bigskip

Now we will take advantage of this example to state a general
result:

\begin{thm}\label{Th5} Let $C$ be a genus $2$ curve defined over $\Q$ by the
equation $y^2=f(x)$, where $f(x)=f_6 x^6 + f_5 x^5 + f_4 x^4 + f_3
x^3 + f_2 x^2 + f_1 x + f_0 \in \Z[x]$ is a polynomial of degree 6
without multiple roots, satisfying that
\begin{equation*}\begin{cases}f_6\equiv f_5\equiv f_3\equiv f_1\equiv
f_0\equiv 1\mod 25\\ f_4\equiv f_2\equiv 0\mod
25.\end{cases}\end{equation*} Then $C$ has stable reduction at $5$,
and this reduction is of type (II). The order of the group of
connected components of the special fibre of the N\'{e}ron model at
$p=5$ is $1$.
\end{thm}

\begin{proof} Due to the congruence condition above, the discriminant of the equation
$y^2=f(x)$ is congruent to the discriminant of the equation $y^2=x^6
+ x^5 + x^3 + x + 1$ modulo 25, that is to say, it is congruent to
$28037120\equiv 20 \mod 25$. Therefore $p=5$ divides the
discriminant of our equation once and only once, thus ensuring that
the curve $C$ has stable reduction. Let us see what type of
reduction it has. Since the invariant $I_{12}$ of the polynomial
$x^6 + x^5 + x^3 + x + 1$ is not divisible by $5$, the same holds
for the invariant $I_{12}$ of $f(x)$ (for both are congruent to each
other modulo 25). Consequently, since the invariants $J_{2i}$ belong
to $\Z_{5}$ (the only denominators which can appear are the powers
of 2), $J_{12}^{2i}I_{12}^{-i}\in \Z_5$. And finally, since $5$ does
divide the discriminant of $f(x)$, that is to say, $J_{12}$, it is
clear that $J_{12}^{10} I_{12}^{-5}$ belongs to the maximal ideal of
$\Z_5$. Theorem \ref{StableReductionII} implies that the reduction
is of type (II).

\end{proof}

\subsection{Control of the image: characteristic polynomials}

Let us now deal with the existence of elements in
$\mathrm{Im}\rho_{\ell}$ with particular characteristic polynomials.
The candidate elements we are going to look at are the images of the
Frobenius elements at primes different from $\ell$, where $C$ has
good reduction, since we have a great deal of information about
their shape.

Namely, let $q\not=\ell$ be a prime number, and assume that $J(C)$
has good reduction at $q$. Inside $G_{\Q}$ we can consider the
decomposition group at $q$, which is isomorphic to
$\Gal(\overline{\Q}_q/\Q_q)$. Different immersions of
$\Gal(\overline{\Q}_q/\Q_q)$ into $G_{\Q}$ give rise to conjugate
subgroups. Consider the Frobenius morphism $x\mapsto x^q$ in
$\Gal(\overline{\F}_{q}/\F_{q})$. There are many liftings of this
element to $\Gal(\overline{\Q}_q/\Q_q)$. Since $C$ has good
reduction at $q$, the image by $\rho_{\ell}$ of any two lifts of the
Frobenius must coincide, and thus an element in $\GSp_4(\F_{\ell})$
is determined save conjugacy. In any case, the characteristic
polynomial of an element in $\GSp_4(\F_{\ell})$ is not altered by
conjugation, so this polynomial is well defined, independently of
any choices we may make along the way. What does this polynomial
look like?

Let us consider the Frobenius endomorphism $\phi_q$ acting on the
reduction $\widetilde{C}$ of $C$ at $q$. It is well known (see for
instance \cite{Cohen-Frey}, $\S$ 14.1.6, Theorem 14.16) that the
characteristic polynomial of $\phi_q$ has the following shape:
\begin{equation*} P(X)=X^4 + a X^3 + b X^2 + a q X + q^2,\end{equation*}
for certain $a, b\in \Z$. More precisely, if we denote by $N_1$
(resp. $N_2$) the number of points of $C$ over $\F_{q}$ (resp.
$\F_{q^2}$), we can compute $a$ and $b$ by using the formulae
\begin{equation}\label{PuntosDeC}\begin{cases}a:=N_1 - q-1\\
                              b:=(N_2 - q^2 - 1 + a^2)/2.\end{cases}\end{equation}

Furthermore, it can be proven (cf. Proposition 10.20 in
\cite{Milne}) that the polynomial obtained from $P(X)$ by reducing
its coefficients modulo $\ell$ coincides with the characteristic
polynomial of $\rho_{\ell}(\Frob_q)$.

In the previous section our problem was to construct a curve such
that, at a certain well chosen prime $p$, it satisfied some
condition. Our strategy there was to choose the prime $p$ once and
for all at the beginning; namely, we took $p=5$, and we established
a congruence condition modulo $5^2$ such that, whenever it is
satisfied by a genus $2$ curve $C$, we achieve our objective. The
first point in which this section differs from the previous one is
that now we will not choose the primes $q_1$ and $q_2$ beforehand;
the primes $q_1$ and $q_2$ will actually depend on $\ell$.

When we face this problem, a natural question arises: Given a prime
$q$, what conditions must a pair $(a, b)$ satisfy in order to ensure
that the polynomial $P(X)=X^4 + aX^3 + bX^2 + qaX + q^2$ is the
characteristic polynomial of the Frobenius endomorphism at $q$ of a
genus $2$ curve?

In \cite{Maisner-Nart}, the authors address the problem of
determining whether, given finite field $\F_q$ and a pair of
positive natural numbers $(N_1, N_2)$, there exists a genus $2$
curve $C$ with $N_1$ points over $\F_q$ and $N_2$ points over
$\F_{q^2}$. We shall make use of their results.

Firstly, let us recall the definition of a Weil polynomial (see
Definition 2.2 of \cite{Maisner-Nart}).

\begin{defn} Let $q$ be a power of a prime number. We will say that the polynomial
$P(X)=X^4 + a X^3 + b X^2 + q a X + q^2\in \Z[X]$ is a \emph{Weil
polynomial} if $\vert a\vert \leq 4\sqrt{q}$ and $2\vert a \vert
\sqrt{q} - 2q \leq b\leq \frac{a^2}{4} + 2q$.
\end{defn}

\begin{rem}To simplify the problem, we will always choose $a=1$, so we will only
have to take care to choose $b$ satisfying $2\sqrt{q} - 2q \leq
b\leq \frac{1}{4} + 2q$.\end{rem}

Let us fix an odd prime number $q$. Collecting Theorem 2.15 and
Theorem 4.3 of \cite{Maisner-Nart}, we can state the following
result:

\begin{thm}\label{Thm MN} Let $P(X)=X^4 + a X^3 + b X^2 + q a X + q^2\in
\Z[X]$ be a Weil polynomial, and let $\Delta_0=a^2 - 4b + 8q$.
Assume that $\Delta_0$ is not a square in $\Z$, $q\nmid b$ and
$a^2\not\in \{0, q + b, 2 b, 3(b - q)\}$. Then there exists a smooth
projective curve of genus $2$, defined over $\F_q$, with $N_1=q + 1
+ a$ points over $\F_q$ and $N_2= 2b - a^2 +q^2 + 1$ points over
$\F_{q^2}$.
\end{thm}

\begin{rem} The previous result claims the existence of a genus
$2$ curve, say $C$, defined over $\F_q$ with $N_1$ points over
$\F_q$ and $N_2$ points over $\F_{q^2}$. If $q$ is odd, we know that
there exists a hyperelliptic equation $y^2=f(x)$ defining $C$, with
$f(x)\in \F_q[x]$ a polynomial of degree $6$ and without multiple
roots. Since there is only a finite number of such polynomials
$f(x)\in \F_q[x]$, one can compute the curve $C$ simply by  an
exhaustive search, so one can say that this construction is
effective. Nevertheless, there are algorithms to compute genus $2$
curves with a given number of points over $\F_q$ and over
$\F_{q^2}$. For instance, see \cite{Eisentrager-Lauter}.
\end{rem}

Keeping this result in mind, the following two propositions show us
how to construct suitable $q_1$ and $q_2$.

\begin{prop}\label{q_1} Let $\ell$ be an odd prime number. Choose $q_1$ such
that $q_1\equiv 1\pmod \ell$. Then there exists a projective curve
of genus $2$, $C_1/\Q$, such that it has good reduction at $q_1$,
and the characteristic polynomial of the Frobenius endomorphism at
$q_1$, $P_1(X)=X^4 + a_1 X^3 + b_1 X^2 + q_1a_1X + q_1^2$ satisfies
that $\Delta_0(P_1)$ is not a square in $\F_{\ell}$ and
$a_1\not\equiv 0 \pmod \ell$.
\end{prop}

\begin{proof} Fix $a_1=1$. Since $q_1\equiv 1\pmod \ell$, it
follows that $q_1>\ell$. Therefore, if we choose any element
$\overline{b}_1\in \F_{\ell}$, there exists $b_1\in \Z$, $0<b_1<q_1$
mapping into $\overline{b}_1$. Therefore $P_1(X)=X^4 + a_1 X^3 + b_1
X^2 + q_1a_1X + q_1^2$ is a Weil polynomial. We will choose
$\overline{b}_1$ such that $1-4\overline{b}_1 + 8q_1$ is not a
square in $\F_{\ell}$ (since $4$ is prime to $\ell$, the expression
$1-4\overline{b}_1 + 8q_1$ runs through all the elements of
$\F_{\ell}$ as $\overline{b}_1$ varies, so this is clearly
feasible).

Now it is easy to check that the pair $(a_1, b_1)$ satisfies all the
conditions in Theorem \ref{Thm MN}, so that there exists a smooth
projective curve of genus $2$ defined over $\F_{q_1}$ with a
suitable number of points over $\F_{q_1}$ and $\F_{q_1^2}$. Lifting
this curve to $\Q$, we obtain the curve we were seeking.
\end{proof}

\begin{prop}\label{q_2} Let $\ell$ be an odd prime number. Choose $q_2$ such
that $q_2\equiv 1\pmod \ell$ and $q_2>3\ell$. Then there exists a
projective curve of genus $2$, $C_2/\Q$, such that it has good
reduction at $q_2$, and the characteristic polynomial of the
Frobenius endomorphism at $q_2$, $P_2(X)=X^4 + a_2 X^3 + b_2 X^2 +
q_2a_2X + q_2^2$ satisfies that $\Delta_0(P_2)$ is a non-zero square
in $\F_{\ell}$ but is not a square in $\Z$, and $P_2(X)$ does not
break up in linear factors over $\F_{\ell}$.
\end{prop}

\begin{proof}
As in the proof of Proposition \ref{q_1}, we will fix $a_2=1$. Note
that, since $q_2>3\ell$, for each element $\overline{b}_2\in
\F_{\ell}$ there exist three values of  $b_2\in \Z, 0<b_2<q_2$ such
that $b_2$ maps into $\overline{b}_2$, which can be taken as $b_2$,
$\ell + b_2$, $2\ell + b_2$.

Let us choose an element $z\in \F_{\ell}$ such that $z^2-16q_2$ is
not a square in $\F_{\ell}$. Such an element exists: if we take any
square $x^2\in \F_{\ell}$, and add $-16q_2$ as many times as we
wish, we can obtain any element in $\F_{\ell}$ that we like. In
particular, if we consider the sequence $x^2, x^2 - 16q_2, x^2 -
2\cdot 16q_2, x^2- 3\cdot 16q_2, \dots$, a point will come when we
obtain a non-square element. The previous element shall be our
$z^2$. If $z^2=1$, we will take $z=1$. Now let us choose $b_2<q_2$
such that $1-4b_2+8q_2$ is congruent to $(z+1)^2$ modulo $\ell$.
This is possible for the same reason as in the proof of Proposition
\ref{q_1}. Moreover, at the beginning of the proof we noted that
there are, in fact, three possible choices for $b_2$ which are
strictly smaller than $q_2$. It is not difficult to check that for
the three of them $1-4b_2+8q_2$ cannot be a square in $\Z$. We have
set this claim aside in Lemma \ref{aside}. Therefore, we can choose
$b_2$ such that $1-4b_2 + 8q_2$ is not a square in $\Z$, and
furthermore $b_2$ is not divisible by $q_2$.

If we choose $b_2$ in this way, it is easy to check that the
conditions of Theorem \ref{Thm MN} hold. Therefore, there exists a
smooth projective genus $2$ curve over $\F_{q_2}$ such that the
characteristic polynomial of the Frobenius endomorphism of $q_2$ is
$P_2(X)=X^4 + a_2 X^3 + b_2 X^2 + q_2a_2X + q_2^2$. Now we ascertain
that the thesis of our Proposition holds. It is clear that
$\Delta_0(P_2)\equiv (z+1)^2\pmod \ell$ is a non-zero square in
$\F_{\ell}$. It remains to show that $P_2(X)$ does not split into
linear factors. Call $\alpha_2, q_2/\alpha_2, \beta_2, q_2/\beta_2$
the roots of $P_2(X)$. The fact that $\Delta_0(P_2)$ is a square
tells us that the polynomials $(X-\alpha_2)(X-q_2/\alpha_2)$ and
$(X-\beta_2)(X-q_2/\beta_2)$ are defined over $\F_{\ell}$. We will
achieve our objective if we see that one of these polynomials is
irreducible over $\F_{\ell}$. Since both $\alpha_2 + q_2/\alpha_2$
and $\beta_2 + q_2/\beta_2$ are roots of $P_0(X)=X^2 + a_2X + b_2 -
(b_2-2q_2)$, they are given by the expressions $\frac{-a_2
\pm\sqrt{a_2^2-4(b_2 -2q_2)}}{2}$. Interchanging $\alpha_2$ and
$\beta_2$ if necessary, we can assume that $\alpha_2+
q_2/\alpha_2=\frac{-a_2 +\sqrt{a_2^2-4(b_2 -2q_2)}}{2}$. Therefore
the polynomial $(X-\alpha_2)(X-q_2/\alpha_2)$ can be written as $X^2
-\frac{-a_2 +\sqrt{a_2^2-4(b_2 -2q_2)}}{2}X + q_2$, and its
discriminant is
\begin{equation*}\Delta=\left(\frac{-a_2+\sqrt{a_2^2-4(b_2-2q_2)}}{2}\right)^2-4q_2.\end{equation*}

Let us compute this quantity modulo $\ell$. Since $a_2=1$,
$\Delta_0=1-4b_2+8q_2\equiv (z+1)^2$, we obtain that $\Delta\equiv
\frac{z-16q_2}{4}\pmod \ell$, which is not a square in $\F_{\ell}$
because of the choice of $z$. This proves that $P_2(X)$ does not
decompose in linear factors over $\F_{\ell}$.

\end{proof}

\begin{lem}\label{aside} Let $A$ be a natural number, and let
$\ell$ be a prime. Then the three numbers $A$, $A-4\ell$, $A-8\ell$
cannot be squares in $\Z$.
\end{lem}

\begin{proof} Assume that there exist $x, y, z$ positive integers such that
$A=x^2$, $A-4\ell=y^2$ and $A-8\ell=z^2$. From the first two
equations we obtain that $4\ell=x^2-y^2=(x+y)(x-y)$. Therefore
$\ell=\frac{x+y}{2}\cdot \frac{x-y}{2}$. Since $\ell$ is a prime
number, it follows that $x-y=2$, and moreover
$\ell=\frac{(y+2)+y}{2}\cdot \frac{(y+2)-y}{2}=y+1$. The same
reasoning applied to the last two equations yields that $y-z=2$, and
we can write $\ell=\frac{y+z}{2}\cdot
\frac{y-z}{2}=\frac{(z+2)+z}{2}\cdot \frac{(z+2)-z}{2}=z+1$. This is
clearly a contradiction.
\end{proof}

\section{Main Result}

In this section we will state the main result concerning tame Galois
realizations of $\GSp_4(\F_{\ell})$. Our starting point is the
following straightforward statement:

\begin{quote}Let $C$ be a smooth projective curve of genus $2$, defined over
$\Q$ and such that, if we denote by $J$ the Jacobian variety
attached to $C$ and by $\rho_{\ell}$ the Galois representation
attached to the $\ell$-torsion points of $J$, the following
conditions are satisfied:

\begin{itemize}

\item The Galois extension obtained by adjoining to $\Q$ the
coordinates of the $\ell$-torsion points of $J$ is tamely ramified.

\item The image of $\rho_{\ell}$ coincides with the general
symplectic group $\GSp_4(\F_{\ell})$.
\end{itemize}

Then $\rho_{\ell}$ provides a tamely ramified Galois realization of
$\GSp_4(\F_{\ell})$.

\end{quote}

Throughout this paper, we have sought to remodel these conditions in
order to make them look like congruences. We have succeeded to a
great extent. Replacing these conditions with those (more
restrictive but simpler) obtained in the previous sections, we
obtain the following result:

\begin{thm}\label{teorema ramificacion 4} Let $C$ be a genus $2$ curve
defined by a hyperelliptic equation
\begin{equation*}y^2=f(x),\end{equation*} where $f(x)=f_6 x^6 + f_5 x^5 + f_4 x^4 + f_3 x^3 + f_2 x^2 + f_1 x + f_0\in \Z[x]$ is a
polynomial of degree $6$ without multiple factors. Let $\ell\geq 7$
be a prime number, and let $\calP$ be the set of prime numbers that
divide the order of $\GSp_4(\F_{\ell})$.

Let $a\in \F_{\ell}$ be such that the elliptic curve defined by
$y^2=x^3 + (1-a)/a x^2 + (1-a)/a x + 1$ is supersingular.
Furthermore, let $q_1, q_2\equiv 1\pmod \ell$ be different prime
numbers with $q_2>3\ell$. Let $C_1/\Q$ be a genus $2$ curve such
that it has good reduction at $q_1$, and the characteristic
polynomial of the Frobenius endomorphism at $q_1$, $P_1(X)=X^4 + a_1
X^3 + b_1 X^2 + q_1a_1X + q_1^2$ satisfies that $\Delta_0(P_1)$ is
not a square in $\F_{\ell}$ and $a_1\not\equiv 0 \pmod \ell$. Let
$C_2/\Q$ be a genus $2$ curve such that it has good reduction at
$q_2$, and the characteristic polynomial of the Frobenius
endomorphism at $q_2$, $P_2(X)=X^4 + a_2 X^3 + b_2 X^2 + q_2a_2X +
q_2^2$ satisfies that $\Delta_0(P_2)$ is a non-zero square in
$\F_{\ell}$ but is not a square in $\Z$, and $P_2(X)$ does not break
up in linear factors over $\F_{\ell}$. Consider hyperelliptic
equations $y^2=c_6 x^6 + c_5 x^5 + c_4 x^4 + c_3 x^3 + c_2 x^2 + c_1
x + c_0$ and $y^2=d_6x^6 + d_5x^5 + d_4x^4 + d_3x^3 + d_2x^2 + d_1x
+ d_0$ defining $C_1$ and $C_2$.

Assume that the following conditions hold:
\begin{itemize}
\item The following congruences mod $2^4$ hold:
\begin{equation*}\begin{cases}f_0\equiv f_6\equiv 1\pmod{16}\\
         f_1\equiv f_5\equiv 0 \pmod{16}\end{cases}
         \begin{cases}f_2\equiv f_4\equiv 4 \pmod{16}\\
         f_3\equiv 2\pmod{16}.\end{cases}\end{equation*}

\item The following congruences mod $3$ hold:
\begin{equation*}\begin{cases}f_0\equiv f_6\equiv 1\pmod
3\\ f_1\equiv f_5\equiv 0\pmod 3\end{cases}
\begin{cases}f_2\equiv f_4\equiv 1\pmod 3\\
f_3\equiv 0 \pmod 3.\end{cases}\end{equation*}

\item The following congruences mod $5^2$ hold:
\begin{equation*}\begin{cases}f_6\equiv f_5\equiv f_3\equiv f_1\equiv
f_0\equiv 1\mod 25\\ f_4\equiv f_2\equiv 0\mod
25.\end{cases}\end{equation*}

\item The following congruences mod $\ell^4$ hold:
\begin{equation*}\begin{cases}f_6\equiv f_0 \pmod{\ell^4}\\
                              f_5\equiv f_1 \pmod{\ell^4}\\
                              f_4\equiv f_2 \pmod{\ell^4}.\end{cases}\end{equation*}
Furthermore, \begin{equation*}\begin{cases}f_6\equiv 1 \pmod{\ell}\\
                              f_5\equiv 0 \pmod{\ell}\end{cases}
                              \begin{cases}f_4\equiv (1-a)/a \pmod{\ell}\\
                              f_3\equiv 0 \pmod{\ell}.\end{cases}\end{equation*}

\item The following congruences mod $q_1$ hold:
\begin{equation*}f_i\equiv c_i \pmod{q_1}, i=0, 1, \dots, 6.\end{equation*}

\item The following congruences mod $q_2$ hold:
\begin{equation*}f_i\equiv d_i \pmod{q_2}, i=0, 1, \dots, 6.\end{equation*}

\item For all $p\in \calP$ different from $2$, $3$, $5$, $q_1$,
$q_2$ and $\ell$, the following congruences hold:
\begin{equation*}\begin{cases}f_0\equiv f_6\equiv 1\pmod
p\\ f_1\equiv f_5\equiv 0\pmod p\end{cases}
\begin{cases}f_2\equiv f_4\equiv 0\pmod p\\
f_3\equiv 0 \pmod p.\end{cases}\end{equation*}
\end{itemize}

Then the Galois representation attached to the $\ell$-torsion points
of the Jacobian of $C$ provides a tamely ramified Galois realization
of $\GSp_4(\F_{\ell})$.
\end{thm}

\begin{proof} The existence of $a\in \F_{\ell}$  follows from Theorem 1-(b) of \cite{Brillhart-Morton} (cf. Corollary 3.6 of \cite{Ariasdereyna-Vila2009}), and the existence of the genus two curves $C_1$ and $C_2$ is proved in Propositions \ref{q_1} and \ref{q_2}. The result is
a direct consequence of the considerations made in the previous
sections.\end{proof}

A quick look at this theorem shows that, for each prime number
$\ell\geq 7$, there exists a genus $2$ curve $C$ satisfying all the
hypotheses, simply because of the Chinese Remainder Theorem.
Therefore, we may write the following corollary:

\begin{cor}\label{Realizacion} For each prime number $\ell\geq 7$,
there exists a Galois extension over $\Q$ which is tamely ramified
and has Galois group $\GSp_4(\F_{\ell})$.
\end{cor}

\begin{rem}\label{El primo 5} As we remarked at the beginning of Section \ref{transvection}, we excluded the prime $\ell=5$ just in order to write a neat
statement for all $\ell\not=5$. Let us now tackle the case $\ell=5$.
Consider the hyperelliptic curve $C$ defined over $\Q$ by the
equation \begin{equation}\label{Eq5} y^2=x^6 + 391300x^4 + 1170 x^3
+ 1300 x^2 + 1.\end{equation} To simplify the notation, call
$f_6=1$, $f_5=0$, $f_4=391300$, $f_3=1170$, $f_2=1300$, $f_1=0$,
$f_0=1$.

We can compute the reduction data of this particular curve by means
of the algorithm of Liu (which is implemented in SAGE). We obtain
that this curve has good reduction outside (possibly) the primes
$2$, $27792683$ and $195476205803858674906021$. In any case, at the
last two primes the conductor exponent is $1$. Therefore the curve
has stable reduction at all odd primes. The algorithm of Liu does
not compute the conductor at the prime $2$. Nevertheless, in this
case it is easy to check that the coefficients $f_i$ of the equation
defining the curve satisfy that
\begin{equation*}\begin{cases}f_6\equiv f_0\equiv 1\pmod{16}\\
         f_5\equiv f_1\equiv 0\pmod{16}\end{cases}
         \begin{cases}f_4\equiv f_2\equiv 4\pmod{16}\\
         f_3\equiv 2\pmod{16}.\end{cases}\end{equation*}

Therefore, Proposition \ref{ReductionAt2} ensures that the reduction
of $C$ at $2$ is also stable. As a conclusion, we can say that the
Galois representation attached to the $5$-torsion points of the
Jacobian surface of $C$ is tamely ramified outside the prime
$\ell=5$.

On the other hand, note that the equation defining $C$ is congruent
modulo $5^4$ with the supersingular symmetric equation $y^2=x^6 +
1300 x^4 + 1170 x^3 + 1300 x^2 + 1$. This guarantees that the wild
inertia group at $5$ acts trivially on the group of $5$-torsion
points.

Let us denote by $\rho_{\ell}:G_{\Q}\rightarrow \GSp_4(\F_{5})$ the
Galois representation which arises from the action of the Galois
group on the points of $5$-torsion of the Jacobian variety attached
to $C$.

The computation of the reduction data of $C$ at the prime
$p=27792683$ shows that the stable reduction at $p$ is of type (II).
Therefore, this prime satisfies the hypothesis of Proposition
\ref{prop toro}. Furthermore the order of the group of connected
components of the special fibre of the N\'{e}ron model at $p$ is
$1$, so it is not divisible by $5$. This ensures the existence of a
transvection in the group
$\mathrm{Im}\rho_{\ell}\subset\GSp_4(\F_{\ell})$. In order to prove
that the image of the Galois representation is $\GSp_4(\F_{\ell})$,
we will make use of Proposition \ref{Prop LeDuff}.

For instance, let us consider the prime $q=19$. The number of points
of $C$ over $\F_{19}$ is $22$, and the number of points of $C$ over
$\F_{19^2}$ is $410$. Therefore, Equations \eqref{PuntosDeC} allow
us to compute the characteristic polynomial of the Frobenius
endomorphism at $q$, namely $P(X)=X^4 + 2X^3 + 26X^2 + 38X + 361$.
It is not difficult to ascertain that this polynomial is irreducible
over $\F_{5}$. Therefore, we conclude that the image of $\rho_5$ is
$\GSp_4(\F_{5})$. Thus this group can also be realized as the Galois
group over $\Q$ of a tamely ramified extension.
\end{rem}

The previous remark allows us to state Corollary \ref{Realizacion}
without excluding the prime $\ell=5$:

\begin{cor}\label{Realizacion2} For each prime number $\ell\geq 5$,
there exists a Galois extension over $\Q$ which is tamely ramified
and has Galois group $\GSp_4(\F_{\ell})$.
\end{cor}

As a matter of fact, we have proven not just the existence of a
tamely ramified Galois realization of $\GSp_4(\F_{\ell})$, but of
infinitely many of them.

\begin{thm}\label{TeoremaPrincipal} For each prime number $\ell\geq 5$, there exist
infinitely many tamely ramified Galois extensions over $\Q$ with
Galois group $\GSp_4(\F_{\ell})$.
\end{thm}

\begin{proof} If $\ell\geq 7$, this is clear from the statement of
Theorem \ref{teorema ramificacion 4}. If $\ell=5$, note that the
cardinal of $\GSp_4(\F_{5})$ equals $37440000=2^9\cdot 3^2\cdot
5^4\cdot 13$. Therefore, each curve $C$ given by a hyperelliptic
equation congruent to \eqref{Eq5} modulo a suitable power of the
primes $2, 3, 5, 13, 27792683$ will provide a tamely ramified Galois
realization of $\GSp_4(\F_{5})$.
\end{proof}

\begin{rem} The symplectic group $\GSp_4(\F_2)$ is isomorphic to the symmetric group $S_6$, thus it is already known to be realizable as the Galois group of a tamely ramified extension of $\Q$ (cf. Proposition 1 in \cite{PlansVila2004}). If $\ell=3$, we hope that a tame Galois realization of $\GSp_4(\F_{3})$ can also be obtained by using different techniques. For instance, one could make use of the isomorphism $\Sp_4(\F_{3})\simeq U(4,2)$ to obtain a regular realization of this group (cf. \cite{Malle-Matzat}), and then apply the techniques in \cite{Plans2003}.
\end{rem}

\section{Example}

In this section we will present an example to illustrate how Theorem
\ref{teorema ramificacion 4} allows us to compute explicitly genus
$2$ curves providing tame Galois realizations of the group
$\GSp_4(\F_{\ell})$. In fact, Theorem \ref{teorema ramificacion 4}
can easily be turned into an algorithm to compute these genus $2$
curves. Even though we have not explicitly formulated it in this
way, we hope this example will clarify how the algorithm works.

Firstly, note the simple well-known result:

\begin{lem}\label{cardinal} Let $q$ be a prime number. Then
\begin{equation*}\card(\GSp_{2n}(\F_q))=(q-1)q^{\frac{(2n)^2}{4}}\prod_{j=1}^{n}
(q^{2j}-1).\end{equation*}
\end{lem}

Let us take $\ell=7$. We will compute all the elements that appear
in the statement of Theorem \ref{teorema ramificacion 4}.

\begin{itemize}

\item First of all, let us compute the set $\calP$ of prime
numbers which divide the order of $\GSp_4(\F_{\ell})$. Since
$\ell=7$ and $n=2$, Lemma \ref{cardinal} yields that
$\card(\GSp_4(\F_{\ell}))=1659571200=2^{10}\cdot 3^{3}\cdot
5^{2}\cdot 7^4$. Therefore, the set of primes that divide the order
of $\GSp_4(\F_{\ell})$ is $\calP=\{2, 3, 5, 7\}$.

\item Next, we need to compute the element $a\in \F_{7}$ such that
the equation $y^2=x^3 + (1-a)/a x^2 + (1-a)/a x + 1$ defines a
supersingular elliptic curve. In order to do this, we compute the
Deuring polynomial
\begin{equation*}H_7(x)=\sum_{j=0}^3\binom{3}{j}^2\cdot
x^j=x^3 + 2x^2 + 2x + 1= (x+1)(x+3)(x+5).\end{equation*} Therefore,
$x^2-x+5=(x+1)(x+5)$ divides the Deuring polynomial, so we can take
$a=5$ in $\F_7$.

\item The following elements that emerge are the primes $q_1$
and $q_2$. We must choose two different prime numbers, $q_1$ and
$q_2$, which are congruent to $1$ modulo $7$ and such that
$q_2>3\cdot 7=21$. We may take $q_1=29$, $q_2=43$. Now we must find
the genus $2$ curves $C_1$ and $C_2$.
\begin{itemize}

\item Choosing the curve $C_1$: firstly, fix $a_1=1$. According to
Proposition \ref{q_1}, we need to find $b_1$ such that $1-4\cdot
b_1+8\cdot 29$ is not a square modulo $\ell$. For instance, we can
take $b_1=1$. Now that we have a pair $(a_1, b_1)$, we seek a genus
$2$ curve over $\F_{29}$ with $N_1=29 + 1 + a_1=31$ points over
$\F_{29}$ and $N_2=29^2 + 1 + 2 b_1 - a_1^2=843$ points over
$\F_{29^2}$. We know that such a curve exists, and if we scrutinize
the set of all hyperelliptic curves defined over $\F_{29}$ (which is
a finite set), we can obtain the curve given by the hyperelliptic
equation $y^2=x^6+ x^5 + 17x + 5$.

\item Choosing the curve $C_2$: again, fix $a_1=1$. According to
Proposition \ref{q_2}, the first step is to find an element $z\in
\F_{7}$ such that  $z^2-16\cdot 43$ is not a square in $\F_{7}$. For
instance, we can consider $z=1$. Next, we must find $b_2$ such that
$1-4\cdot b_2+8\cdot 43$ is congruent to $(z+1)^2=4$ modulo $7$. For
instance, take $b_2=3$. Note that $1-4\cdot 3 +8\cdot 43=333$ is not
a square in $\Z$. We have a pair $(a_2, b_2)$. We must seek a genus
$2$ curve over $\F_{43}$ with $N_1=43 + 1 + a_2=45$ points over
$\F_{43}$ and $N_2=43^2 + 1 + 2 b_2 - a_2^2=1855$ points over
$\F_{43^2}$. Such a curve exists, and again an exhaustive search can
provide it. For instance, we have taken the curve defined by the
hyperelliptic equation $y^2=x^6 + x^5 + 3x^2 + 13x + 21$.
\end{itemize}

\end{itemize}

Let us now go through Theorem \ref{teorema ramificacion 4} replacing
the elements that appear there by the ones we have chosen above:

\begin{prop} Let $C$ be a genus $2$ curve
defined by a hyperelliptic equation
\begin{equation*}y^2=f(x),\end{equation*}
where $f(x)=f_6x^6 + f_5 x^5 + f_4 x^4 + f_3 x^3 + f_2 x^2 + f_1 x +
f_0\in \Z[x]$ is a polynomial of degree 6 without multiple factors.
Assume that the following conditions hold:
\begin{itemize}
\item The following congruences mod $2^4$ hold:
\begin{equation*}\begin{cases}f_0\equiv f_6\equiv 1\pmod{16}\\
         f_1\equiv f_5\equiv 0 \pmod{16}\end{cases}
         \begin{cases}f_2\equiv f_4\equiv 4 \pmod{16}\\
         f_3\equiv 2\pmod{16}.\end{cases}\end{equation*}

\item The following congruences mod $3$ hold:
\begin{equation*}\begin{cases}f_0\equiv f_6\equiv 1\pmod{3}\\
         f_1\equiv f_5\equiv 0 \pmod{3}\end{cases}
         \begin{cases}f_2\equiv f_4\equiv 1 \pmod{3}\\
         f_3\equiv 0\pmod{3}.\end{cases}\end{equation*}

\item The following congruences mod $5^2$ hold:
\begin{equation*}\begin{cases}f_6\equiv f_5\equiv f_3\equiv f_1\equiv
f_0\equiv 1\mod 25\\ f_4\equiv f_2\equiv 0\mod
25.\end{cases}\end{equation*}

\item The following congruences mod $7^4$ hold:
\begin{equation*}\begin{cases}f_6\equiv f_0 \pmod{7^4}\\
                              f_5\equiv f_1 \pmod{7^4}\\
                              f_4\equiv f_2 \pmod{7^4}.\end{cases}\end{equation*}
Furthermore, \begin{equation*}\begin{cases}f_6\equiv 1 \pmod{7}\\
                              f_5\equiv 0 \pmod{7}\end{cases}
                              \begin{cases}f_4\equiv 2 \pmod{7}\\
                              f_3\equiv 0 \pmod{7}.\end{cases}\end{equation*}

\item The following congruences mod $29$ hold:
\begin{equation*}\begin{cases}f_6\equiv 1 \pmod{29}\\
                              f_5\equiv 1 \pmod{29}\\
                              f_4\equiv 0 \pmod{29}\\
                              f_3\equiv 0 \pmod{29}\end{cases}
                              \begin{cases}f_2\equiv 0 \pmod{29}\\
                              f_1\equiv 17 \pmod{29}\\
                              f_0\equiv 5 \pmod{29}.\end{cases}\end{equation*}

\item The following congruences mod $43$ hold:
\begin{equation*}\begin{cases}f_6\equiv 1 \pmod{43}\\
                              f_5\equiv 1 \pmod{43}\\
                              f_4\equiv 0 \pmod{43}\\
                              f_3\equiv 0 \pmod{43}\end{cases}
                              \begin{cases}f_2\equiv 3 \pmod{43}\\
                              f_1\equiv 13 \pmod{43}\\
                              f_0\equiv 21 \pmod{43}.\end{cases}\end{equation*}

\end{itemize}

Then the Galois extension $\Q(J(C)[7])/\Q$ provides a tamely
ramified Galois realization of $\GSp_4(\F_{7})$.
\end{prop}

It is easy to construct infinitely many such curves. For instance,
we may take the genus $2$ curve defined by the hyperelliptic
equation
\begin{multline*}y^2 = x^6 + 9757776\cdot x^5 + 8853700\cdot x^4 + 10422426\cdot
x^3 + \\
 +677292100\cdot x^2 + 3179077776\cdot x +
342862800.\end{multline*}

\bibliographystyle{amsalpha}

\end{document}